\documentclass[12pt, reqno]{amsart}
\usepackage{amsmath, amsthm, amscd, amsfonts, amssymb, graphicx, color}
\usepackage[bookmarksnumbered, colorlinks, plainpages]{hyperref}
\hypersetup{colorlinks=true,linkcolor=red, anchorcolor=green, citecolor=cyan, urlcolor=red, filecolor=magenta, pdftoolbar=true}

\newtheorem{theorem}{Theorem}[section]
\newtheorem{lemma}[theorem]{Lemma}
\newtheorem{proposition}[theorem]{Proposition}

\theoremstyle{definition}

\theoremstyle{remark}

\numberwithin{equation}{section}

\begin{document}

\setcounter{page}{1}

\title[A note on stability of Hardy inequalities]{A note on stability of Hardy inequalities}

\author[Michael Ruzhansky \MakeLowercase{and} Durvudkhan Suragan]{Michael Ruzhansky,$^1$  \MakeLowercase{and} Durvudkhan Suragan$^2$$^{*}$}

\address{$^{1}$	Department of Mathematics,
		Imperial College London,
	180 Queen's Gate, London SW7 2AZ,
	United Kingdom}
\email{\textcolor[rgb]{0.00,0.00,0.84}{m.ruzhansky@imperial.ac.uk}}

\address{$^{2}$  Institute of Mathematics and Mathematical Modelling,
	125 Pushkin str.,
	050010 Almaty,
	Kazakhstan.}
\email{\textcolor[rgb]{0.00,0.00,0.84}{suragan@math.kz}}


\let\thefootnote\relax\footnote{Copyright 2016 by the Tusi Mathematical Research Group.}



\date{Received: xxxxxx; Revised: yyyyyy; Accepted: zzzzzz.
\newline \indent $^{*}$Corresponding author}

\begin{abstract}
In this note we formulate recent stability results for Hardy inequalities in the language of Folland and Stein's homogeneous groups. Consequently, we obtain remainder estimates for Rellich type inequalities on homogeneous groups. Main differences from the Euclidean results are that the obtained stability estimates hold for any homogeneous quasi-norm.
\end{abstract} \maketitle

\section{Introduction}
\label{SEC:intro}

Recall the  $L^{p}$-Hardy inequality
\begin{equation}\label{Lp_Hardy}
\int_{\mathbb{R}^{n}}|\nabla f|^{p}dx\geq\left(\frac{n-p }{p}\right)^{p}\int_{\mathbb{R}^{n}}\frac{|f|^{p}}{|x|^{p}}dx
\end{equation}
for every function $f\in C_{0}^{\infty}(\mathbb{R}^{n})$, where $2\leq p<n$. 

Cianchi and Ferone \cite {CF08} showed that for all $1<p<n$ there exists a constant $C=C(p,n)$ such that
$$\int_{\mathbb{R}^{n}}|\nabla f|^{p}dx\geq\left(\frac{n-p}{p}\right)^{p}\int_{\mathbb{R}^{n}}\frac{|f|^{p}}{|x|^{p}}dx\,(1+Cd_{p}(f)^{2p^{*}})$$
holds for all real-valued weakly differentiable functions $f$ in $\mathbb{R}^{n}$ such that $f$ and $|\nabla f|\in L^{p}(\mathbb{R}^{n})$ go to zero at infinity. Here
$$d_{p}f=\underset{c\in \mathbb{R}}{\rm inf}\frac{\|f-c|x|^{-\frac{n-p}{p}}\|_{L^{p^{*},\infty}(\mathbb{R}^{n})}}{\|f\|_{L^{p^{*},p}(\mathbb{R}^{n})}}$$ with $p^{*}=\frac{np}{n-p}$, and $L^{\tau, \sigma}(\mathbb{R}^{n})$ is the Lorentz space for $0<\tau\leq \infty$ and $1\leq\sigma\leq\infty$. Sometimes the improved versions of different inequalities, or remainder estimates, are called stability of the inequality if the estimates depend on certain distances: see, e.g. \cite{BJOS16} for stability of trace theorems, \cite{CFW13} for stability of Sobolev inequalities, etc. For more general Lie group discussions of above inequalities we refer to recent papers \cite{Ruzhansky-Suragan:Layers}, \cite{RS-identities} and \cite{Ruzhansky-Suragan:squares} as well as references therein.

Recently Sano and Takahashi obtained the improved versions of Hardy inequalities in their works \cite{S17}, \cite{ST15a}, \cite{ST15b} and \cite{ST15}. The aim of this note is to formulate their results one of the largest classes of nilpotent Lie groups on $\mathbb{R}^{n}$, namely, homogeneous Lie groups since obtained results give new insights even for the Abelian groups in term of arbitrariness of homogeneous quasi-norm.

\section{Preliminaries}
\label{SEC:prelim}

First let us shortly review some main concepts of homogeneous
groups following Folland and Stein \cite{FS-Hardy} (see also recent books \cite{BLU07} and \cite{FR} on this topic). We also recall a few other facts that will be used in the proofs.
A connected simply connected Lie group $\mathbb G$ is called a {\em homogeneous group} if
its Lie algebra $\mathfrak{g}$ is equipped with a family of the following dilations:
$$D_{\lambda}={\rm Exp}(A \,{\rm ln}\lambda)=\sum_{k=0}^{\infty}
\frac{1}{k!}({\rm ln}(\lambda) A)^{k},$$
where $A$ is a diagonalisable positive linear operator on $\mathfrak{g}$,
and every $D_{\lambda}$ is a morphism of $\mathfrak{g}$,
that is,
$$\forall X,Y\in \mathfrak{g},\, \lambda>0,\;
[D_{\lambda}X, D_{\lambda}Y]=D_{\lambda}[X,Y],$$
holds. We recall that $Q := {\rm Tr}\,A$ is called the homogeneous dimension of $\mathbb G$. The Haar measure on a homogeneous group $\mathbb{G}$ is the standard Lebesgue measure for $\mathbb{R}^{n}$ (see, for example \cite[Proposition 1.6.6]{FR}).

Let $|\cdot|$ be a homogeneous quasi-norm on $\mathbb G$.
Then the quasi-ball centred at $x\in\mathbb{G}$ with radius $R > 0$ is defined by
$$B(x,R):=\{y\in \mathbb{G}: |x^{-1}y|<R\}.$$
We refer to \cite{FS-Hardy} for the proof of the following important polar decomposition on homogeneous Lie groups, which can be also found in \cite[Section 3.1.7]{FR}:
there is a (unique)
positive Borel measure $\sigma$ on the
unit quasi-sphere
\begin{equation}\label{EQ:sphere}
\wp:=\{x\in \mathbb{G}:\,|x|=1\},
\end{equation}
so that for every $f\in L^{1}(\mathbb{G})$ we have
\begin{equation}\label{EQ:polar}
\int_{\mathbb{G}}f(x)dx=\int_{0}^{\infty}
\int_{\wp}f(ry)r^{Q-1}d\sigma(y)dr.
\end{equation}

We use the notation
\begin{equation}\label{dfdr}
\mathcal{R} f(x):= 	\mathcal{R}_{|x|} f(x) = \frac{d}{d|x|}f(x)=\mathcal{R}f(x), \;\forall x\in \mathbb G,
\end{equation}
for any homogeneous quasi-norm $|x|$ on $\mathbb G$.

We will also use the following result:

\begin{lemma}[\cite{Ruzhansky-Suragan:critical}]
	\label{CKN}
	Let $\mathbb{G}$ be a homogeneous group
	of homogeneous dimension $Q$. Let $|\cdot|$ be any homogeneous norm on  $\mathbb{G}$.  Then for $u\in C_{0}^{\infty}(\mathbb{G}\backslash\{0\})$ and $u_{R}=u\left(R\frac{x}{|x|}\right)$ we have
	\begin{equation}\label{ScalHardy4}
	\left\|\frac{u-u_{R}}{|x|^{\frac{Q}{p}}\log\frac{R}{|x|}}\right\|_{L^{p}(\mathbb{G})}
	\leq\frac{p}{p-1}\left\||x|^{\frac{p-Q}{p}}
	\mathcal{R}u \right\|_{L^{p}(\mathbb{G})}, \;\;1<p<\infty,
	\end{equation}
	for all $R>0$, and the constant $\frac{p}{p-1}$ is sharp.
\end{lemma}

In the abelian isotropic case, the following result was obtained in \cite{MOW15}. In the case $\gamma=p$ this result on the homogeneous group was proved in \cite{Ruzhansky-Suragan:critical}.

We will also use the following known relations  
\begin{lemma}\label{ab_relation}
	Let $a,b\in \mathbb{R}$. Then
	\begin{itemize}
		\item[i.] 	$$|a-b|^{p}-|a|^{p}\geq-p|a|^{p-2}ab,\quad p\geq 1.$$ 
		\item[ii.] 	There exists a constant $C=C(p)>0$ such that 	
		$$|a-b|^{p}-|a|^{p}\geq-p|a|^{p-2}ab+C|b|^{p},\quad p\geq 2.$$
		\item[iii.] If $a\geq 0$ and $a-b\geq 0$. Then 
		$$(a-b)^{p}+pa^{p-1}b-a^{p}\geq |b|^{p},\quad p\geq 2.$$
	\end{itemize}
\end{lemma}

\section{Stability of $L^{p}$-Hardy inequalities}

Let us set
\begin{equation*}
d_H (u;R) := \left(\int_{\mathbb{G}} \frac{\left|u(x)-R^{\frac{Q-p}{p}}u\left(R\frac{x}{|x|}\right)|x|^{-\frac{Q-p}{p}}\right|^p}{|x|^p|\log\frac{R}{|x|}|^p} dx \right)^{\frac{1}{p}},\; x\in\mathbb{G},\; R>0.
\end{equation*}

\begin{theorem}\label{stab_H}
	Let $\mathbb{G}$ be a homogeneous group
	of homogeneous dimension $Q$. Let $|\cdot|$ be any homogeneous quasi-norm on $\mathbb{G}$. Then there exists a constant $C>0$  for all real-valued functions $u\in C_{0}^{\infty}(\mathbb{G})$ we have 
	\begin{equation}\label{stab_H_eq}
	\int_{\mathbb{G}} \left|\mathcal{R} u \right|^p dx - \left(\frac{Q-p}{p}\right)^p \int_{\mathbb{G}} \frac{|u|^p}{|x|^p} dx \geq C \sup_{R>0} d^p_H(u;R),\; 2 \leq p< Q,
	\end{equation}
	where $\mathcal{R}:=\frac{d}{d|x|}$ is the radial derivative.
\end{theorem}

\begin{proof}[Proof of Theorem \ref{stab_H}]
	
	Let us introduce polar coordinates $x=(r,y)=(|x|, \frac{x}{\mid x\mid})\in (0,\infty)\times\wp$ 
	on $\mathbb{G}$, where $\wp$ is the
	unit quasi-sphere
	\begin{equation}
	\wp:=\{x\in \mathbb{G}:\,|x|=1\},
	\end{equation} and 
	
	\begin{equation}\label{trasn}
	v(ry):= r^{\frac{Q-p}{p}}u(ry), 
	\end{equation} 
	where $u\in C_{0}^{\infty}(\mathbb{G})$.	
	This follows that $v(0)=0$ and $\underset{r\rightarrow  \infty }{\lim}v(ry)=0$ for $y \in  \wp$ since $u$ is compactly supported. Using the polar decomposition on homogeneous groups (see \eqref{EQ:polar}) and integrating by parts, we get
	\begin{align*}
	D:&= \int_{\mathbb{G}} \left|\mathcal{R} u \right|^p dx - \left(\frac{Q-p}{p}\right)^p \int_{\mathbb{G}} \frac{|u|^p}{|x|^p} dx \\
	& = \int_{ \wp}\int_{0}^{\infty} \left| -\frac{\partial }{\partial r} u(ry)\right|^p r^{Q-1} - \left(\frac{Q-p}{p}\right)^p |u(ry)|^p r^{Q-p-1} drdy \\
	& = \int_{ \wp}\int_{0}^{\infty} \left| \frac{Q-p}{p} r^{-\frac{Q}{p}} v(ry) - r^{-\frac{Q-p}{p}} \frac{\partial}{\partial r} v(ry)\right|^p r^{Q-1} \\ & - \left(\frac{Q-p}{p}\right)^p |v(ry)|^p r^{-1} drdy.
	\end{align*}
	Now using the second relation in Lemma \ref{ab_relation}  with the choice $a = \frac{Q-p}{p}r^{-\frac{Q}{p}}v(ry)$ and $b = r^{-\frac{Q-p}{p}}\frac{\partial }{\partial r}v(ry)$, and using the fact $\int_{0}^{\infty}|v|^{p-2}v\left(\frac{\partial}{\partial r}v\right)dr =0$, we obtain
	\begin{align}\label{1_3.2} 
	D &\geq \int_{ \wp}\int_{0}^{\infty}  -p \left(\frac{Q-p}{p} \right)^{p-1} |v(ry)|^{p-2} v(ry) \frac{\partial }{\partial r} v(ry) 
	\\ & +C\left|\frac{\partial}{\partial r}v(ry)\right|^p r^{p-1} drdy \\
	& = C \int_{\mathbb{G}} |x|^{p-Q} \left| \mathcal{R} v \right|^p dx. \nonumber
	\end{align} 
	Finally, combining \eqref{1_3.2} and Lemma \ref{CKN}, we arrive at  
	\begin{align}
	D &\geq C \int_{\mathbb{G}} \frac{|v(x) -v(R\frac{x}{|x|})|^p}{|x|^Q |\log \frac{R}{|x|}|^p} dx = C \int_{ \wp}\int_{0}^{\infty}  \frac{|v(ry) -v(Ry)|^p}{r\left|\log \frac{R}{r}\right|^p} drdy \\
	& =C\int_{ \wp}\int_{0}^{\infty} \frac{|u(ry)-R^{\frac{Q-p}{p}}u(Ry)r^{-\frac{Q-p}{p}}|^p}{r^{1+p-Q}|\log \frac{R}{r}|^p} drdy \nonumber
	\end{align}
	for any $R >0$. This proves the desired result.
\end{proof}
\section{Stability of critical Hardy inequalities}
In this section we establish a stability estimate for the critical Hardy inequality involving the distance to the set of extremisers:
Let us denote
\begin{equation}\label{aremterm7}
f_{T,R}(x)=T^{\frac{Q-1}{Q}}u\left(Re^{-\frac{1}{T}}\frac{x}{|x|}\right)\left(\log \frac{R}{|x|}\right)^{\frac{Q-1}{Q}}
\end{equation}
and the following 'distance'
\begin{equation}\label{aremterm8}
d_{cH}(u;T, R):= \left( \int_{B(0,R)} \frac{|u(x)-f_{T,R}(x)|^Q}{|x|^Q\left|\log \frac{R}{|x|}\right|^Q\left|T\log \frac{R}{|x|}\right|^Q} dx\right)^{\frac{1}{Q}},
\end{equation}
for some parameter $T>0$, functions $u$ and $f_{T,R}$ for which the integral in \eqref{aremterm8} is finite.
\begin{theorem}\label{stab_CH}
	Let $\mathbb{G}$ be a homogeneous group
	of homogeneous dimension $Q\geq2$. Let $|\cdot|$ be any homogeneous quasi-norm on $\mathbb{G}$. Then there exists a constant $C>0$ for all real-valued functions $u\in C_{0}^{\infty}(B(0,R))$ we have 
	\begin{equation}\label{aremterm9}
	\int_{B(0,R)} \left| \mathcal{R} u(x) \right|^Q dx - \left(\frac{Q-1}{Q}\right)^{Q} \int_{B(0,R)} \frac{|u(x)|^Q}{|x|^Q (\log \frac{R}{|x|})^Q}dx \geq C \sup_{T>0} d^Q_{cH} (u;T, R)
	\end{equation}
	where $\mathcal{R}:=\frac{d}{d|x|}$ is the radial derivative.
\end{theorem}

\begin{proof}[Proof of Theorem \ref{stab_CH}] 
	Introducing polar coordinates $(r,y)=(|x|, \frac{x}{\mid x\mid})\in (0,\infty)\times\wp$ on $\mathbb{G}$, where $\wp$ is the sphere as in \eqref{EQ:sphere}, we have $u(x)=u(ry)\in C_0^{\infty}(B(0,R))$. In addition, let us set   
	\begin{equation}\label{5.4}
	v (sy): = \left( \log \frac{R}{r}\right)^{-\frac{Q-1}{Q}}u(ry),\; y \in \wp,
	\end{equation}
	where 
	\begin{equation*}
	s=s(r):=\left( \log \frac{R}{r}\right)^{-1}.
	\end{equation*}
	Since $u \in C_0^{\infty}(B(0,R))$ we have $v(0)=0$ and $v$ has a compact support. Moreover, it is straightforward that
	\begin{equation*}
	\frac{\partial }{\partial r} u(ry) = - \left(\frac{Q-1}{Q}\right)\left(\log \frac{R}{r}\right)^{-\frac{1}{Q}}\frac{v(sy)}{r} + \left(\log \frac{R}{r}\right)^{\frac{Q-1}{Q}}\frac{\partial}{\partial s} v(sy)s'(r).
	\end{equation*}
	A direct calculation gives  
	\begin{align*}\label{5.5}
	S: & = \int_{B(0,R)} \left| \mathcal{R} u \right|^Q dx - \left(\frac{Q-1}{Q}\right)^Q \int_{B(0,R)} \frac{|u|^Q}{|x|^Q \left(\log \frac{R}{|x|}\right)^Q} dx
	\\ & = \int_{\wp} \int_{0}^{R} \left|\frac{\partial }{\partial r}u(ry) \right|^Q r^{Q-1} - \left(\frac{Q-1}{Q}\right)^Q \frac{|u(ry)|^Q}{r\left(\log \frac{R}{r}\right)^Q} dr dy \\
	&= \int_{\wp} \int_{0}^{R} \left| \left(\frac{Q-1}{Q}\right)\left(r\log \frac{R}{r}\right)^{-\frac{1}{Q}}v(sy) 
	 + \left(r\log \frac{R}{r}\right)^{\frac{Q-1}{Q}}\frac{\partial}{\partial s} v(sy)s'(r) \right|^Q \\
	&- \left(\frac{Q-1}{Q}\right)^Q \frac{|v(sy)|^Q}{r\log \frac{R}{r}} dr dy.  	
	\end{align*}
	Now by applying the second relation in Lemma \ref{ab_relation} with the choice 
	\begin{equation*}
	a = \frac{Q-1}{Q} \left(r \log \frac{R}{r}\right)^{-\frac{1}{Q}}v(sy) \quad \text{and} \quad b = \left(r \log \frac{R}{r}\right)^{\frac{Q-1}{Q}} \frac{\partial}{\partial s} v(sy)s'(r),
	\end{equation*}
	and by using the facts $v(0)=0$ and $\underset{r\rightarrow \infty}{\lim}v(ry)=0$, we obtain 
	\begin{align*}
	S & \geq \int_{\wp} \int_{0}^{R} -Q \left(\frac{Q-1}{Q}\right)^{Q-1} |v(sy)|^{Q-2}v(sy) \frac{\partial }{\partial s}v(sy)s'(r) \nonumber \\
	& + C \left|\frac{\partial }{\partial s}v(sy) \right|^Q (s'(r))^Q \left(r \log \frac{R}{r}\right)^{Q-1} drdy \nonumber \\ 
	& =\int_{\wp} \int_{0}^{R} -Q \left(\frac{Q-1}{Q}\right)^{Q-1} |v(sy)|^{Q-2}v(sy) \frac{\partial }{\partial s}v(sy)s'(r) \nonumber \\
	& + C \left|\frac{\partial }{\partial s}v(sy) \right|^Q \frac{1}{r^{Q}\left( \log \frac{R}{r}\right)^{2Q}} \left(r \log \frac{R}{r}\right)^{Q-1} drdy \nonumber \\ 
	& =\int_{\wp} \int_{0}^{R} -Q \left(\frac{Q-1}{Q}\right)^{Q-1} |v(sy)|^{Q-2}v(sy) \frac{\partial }{\partial s}v(sy)s'(r) \nonumber \\
	& + C \left|\frac{\partial }{\partial s}v(sy) \right|^Q \frac{1}{\left( \log \frac{R}{r}\right)^{Q-1}} s'(r) drdy \nonumber \\ 
	& = \int_{\wp} \int_{0}^{R} -Q \left(\frac{Q-1}{Q}\right)^{Q-1} |v(sy)|^{Q-2}v(sy) \frac{\partial }{\partial s} v(s) 
\\	&  +C  \left|\frac{\partial }{\partial s}v(sy) \right|^Q s^{Q-1}dsdy \nonumber \\
	& = C \int_{\mathbb{G}} \left| \mathcal{R} v \right|^Q dx, \nonumber 
	\end{align*} 
	that is, 
	\begin{equation}\label{5.6}
	S \geq C \int_{\mathbb{G}} \left| \mathcal{R} v \right|^Q dx.
	\end{equation}
	
	According to Lemma \ref{CKN} with $v \in C^{\infty}_0(\mathbb{G}\backslash\{0\})$ with $p=Q$ and \eqref{5.6}, it implies that 
	\begin{align*}
	S &\geq C \int_{\mathbb{G}} \frac{|v(x) - v(T\frac{x}{|x|})|^Q}{|x|^Q|\log \frac{T}{|x|}|^Q} dx = C \int_{\wp} \int_{0}^{\infty} \frac{|v(sy)-v(Ty)|^Q}{s|\log \frac{T}{s}|^Q} ds dy \\
	& = C \int_{\wp} \int_{0}^{R} \frac{\left|\left(\log \frac{R}{r}\right)^{-\frac{Q-1}{Q}}u(ry) - T^{\frac{Q-1}{Q}} u(Re^{-\frac{1}{T}}y) \right|^Q}{r(\log \frac{R}{r})|\log (T\log\frac{R}{r} )|^Q} dr dy \\ 
	& = C \int_{\wp} \int_{0}^{R}  \frac{\left|u(ry) - T^{\frac{Q-1}{Q}} u(Re^{-\frac{1}{T}}y)(\log \frac{R}{r})^{\frac{Q-1}{Q}} \right|^Q}{r(\log \frac{R}{r})^Q|\log (T\log\frac{R}{r} )|^Q} dr dy.
	\end{align*}
	Thus, we arrive at   
	\begin{equation*}
	S \geq C \int_{B(0,R)} \frac{\left|u(x)-T^{\frac{Q-1}{Q}} u\left(Re^{-\frac{1}{T}}\frac{x}{|x|}\right)\left(\log \frac{R}{|x|}\right)^{\frac{Q-1}{Q}}\right|^Q}{|x|^Q\left|\log \frac{R}{|x|}\right|^Q \left|\log \left(T \log \frac{R}{|x|}\right)\right|^Q} dx
	\end{equation*}
	for all $T>0$. The proof is  complete.
\end{proof}
\section{Improved critical Hardy and Rellich inequalities for radial functions}

\begin{proposition}\label{radial_rem}
	Let $\mathbb{G}$ be a homogeneous group
	of homogeneous dimension $Q\geq2$. Let $|\cdot|$ be a homogeneous quasi-norm on $\mathbb{G}$. Let $q>0$ be such that
	\begin{equation}\label{1_1.5}
	\alpha = \alpha(q,L): = \frac{Q-1}{Q}q +L +2 \leq Q,
	\end{equation}
	for $-1 <L<Q-2$. 
	Then for all real-valued positive non-increasing radial functions $u\in C_0^{\infty}(B(0,R))$ we have 
	\begin{align}\label{1.6}
	\int_{B(0,R)} |\mathcal{R}  u|^Q dx &- \left( \frac{Q-1}{Q}\right)^Q \int_{B(0,R)} \frac{|u(x)|^Q}{|x|^Q \left(\log \frac{R e}{|x|}\right)^Q} dx \\
	& \geq |\wp|^{1-\frac{Q}{q}} C^{\frac{Q}{q}} \left(\int_{B(0,R)}\frac{|u(x)|^q}{|x|^Q \left(\log \frac{R e}{|x|}\right)^{\alpha}} dx \right)^{\frac{Q}{q}}, \nonumber
	\end{align}
	where  $|\wp| $ is the measure of the unit quasi-sphere in $\mathbb{G}$ and
	\begin{align*}
	C^{-1}=C(L,Q,q)^{-1} & := \int_{0}^{1} s^L \left(\log \frac{1}{s}\right)^{\frac{Q-1}{Q}q} ds \\ &= (L+1)^{-\left(\frac{Q-1}{Q}q+1\right)}\Gamma \left(\frac{Q-1}{Q}q +1\right)
	\end{align*}
	here $\Gamma (\cdot)$ is the Gamma function.
\end{proposition}
\begin{proof}[Proof of Proposition \ref{radial_rem}]
	As in previous proofs we set  
	\begin{align}\label{2.2}
	&v(s) = \left(\log \frac{R e}{r}\right)^{-\frac{Q-1}{Q}} u(r), \quad \text{where} \quad r =|x|,s=s(r)=\left(\log \frac{R e}{r}\right)^{-1}, \\
	&s'(r) = \frac{s(r)}{r \log \frac{Re}{r}} \geq 0. \nonumber
	\end{align}
	Simply we have $v(0)=v(1)=0$ since $u(R)=0$, moreover, 
	\begin{multline}\label{2.3}
	u'(r) =-\left(\frac{Q-1}{Q}\right)\left(\log \frac{R e}{r}\right)^{-\frac{1}{Q}} \frac{v(s(r))}{r} \\+ \left(\log \frac{R e}{r}\right)^{\frac{Q-1}{Q}} v'(s(r))s'(r)\leq 0.
	\end{multline}
	It is straightforward that 
	\begin{align*}
&	I:= \int_{B(0,R)} |\mathcal{R}  u|^Q dx - \left(\frac{Q-1}{Q}\right)^Q \int_{B(0,R)} \frac{|u|^Q}{|x|^Q \left(\log \frac{R e}{|x|}\right)^Q}dx \\ 
	& = |\wp| \int_{0}^{R} |u'(r)|^Q r^{Q-1} dr - \left(\frac{Q-1}{Q}\right)^Q |\wp| \int_{0}^R \frac{|u(r)|^Q}{r \left(\log \frac{R e}{r}\right)^Q}dr
	\\
	& = |\wp| \int_{0}^R \left( \frac{Q-1}{Q} \left(\log \frac{R e}{r}\right)^{-\frac{1}{Q}} \frac{v(s(r))}{r} - \left(\log \frac{R e}{r}\right)^{\frac{Q-1}{Q}} v'(s(r))s'(r)\right)^Q r^{Q-1}dr
	\\
	& - \left(\frac{Q-1}{Q}\right)^Q |\wp| \int_{0}^R \frac{|u(r)|^Q}{r \left(\log \frac{R e}{r}\right)^Q}dr.	 
	\end{align*}
	By applying the third relation in Lemma \ref{ab_relation} with 
	\begin{equation*}
	a = \frac{Q-1}{Q} \left( \log \frac{Re}{r} \right)^{-\frac{1}{Q}} \frac{v(s(r))}{r} \quad \text{and} \quad b = \left( \log \frac{Re}{r} \right)^{\frac{Q-1}{Q}} v'(s(r))s'(r),
	\end{equation*}
	and	dropping $a^Q \geq 0$  as well as using the boundary conditions $v(0)=v(1)=0$, we get 
	\begin{align}\label{2.5}
	I &\geq -|\wp| Q\left(\frac{Q-1}{Q}\right)^{Q-1} \int_0^R v(s(r))^{Q-1}v'(s(r))s'(r)dr \\ 
	& + |\wp| \int_{0}^{R} |v'(s(r))|^Q (s'(r))^Q \left(r \log \frac{Re}{r} \right)^{Q-1} dr\nonumber \\
	& =-|\wp| Q\left(\frac{Q-1}{Q}\right)^{Q-1} \int_0^R v(s(r))^{Q-1}v'(s(r))s'(r)dr \\ 
	& + |\wp| \int_{0}^{R} |v'(s(r))|^Q \frac{1}{r^{Q}\left(\log \frac{Re}{r} \right)^{2Q}} \left(r \log \frac{Re}{r} \right)^{Q-1} dr\nonumber \\
	&=-|\wp| Q\left(\frac{Q-1}{Q}\right)^{Q-1} \int_0^R v(s(r))^{Q-1}v'(s(r))s'(r)dr \\ 
	& + |\wp| \int_{0}^{R} |v'(s(r))|^Q s(r)^{Q-1}s'(r) dr\nonumber \\
	&= - |\wp| Q \left(\frac{Q-1}{Q}\right)^{Q-1}  \int_{0}^{1} v(s)^{Q-1} v'(s)ds 
	\\ & + |\wp| \int_{0}^{1} |v'(s)|^Qs^{Q-1}ds \nonumber\\
	& =|\wp| \int_{0}^{1} |v'(s)|^Q s^{Q-1}ds.\nonumber
	\end{align}
	Moreover, by using the inequality  
	\begin{align*}
	|v(s)|= \left| \int_{s}^{1} v'(t)dt \right| & = \left| \int_{s}^{1} v'(t) t^{\frac{Q-1}{Q} -\frac{Q-1}{Q}} dt\right| \\ & \leq \left(\int_{0}^{1} |v'(t)|^Qt^{Q-1}dt\right)^{\frac{1}{Q}} \left(\log \frac{1}{s}\right)^{\frac{Q-1}{Q}},
	\end{align*}
	we obtain
	\begin{equation*}
	\int_{0}^{1} |v(s)|^q s^Lds \leq \left(\int_{0}^{1} |v'(s)|^Q s^{Q-1} ds \right)^{\frac{q}{Q}} \int_{0}^{1} s^L \left(\log \frac{1}{s}\right)^{\frac{Q-1}{Q}q} ds
	\end{equation*}
	for $-1 <L<Q-2$.
	Thus, we have
	\begin{equation}\label{2.6}
	\int_{0}^{1} |v'(s)|^Q s^{Q-1} ds \geq C^{\frac{q}{Q}} \left(\int_{0}^{1} |v(s)|^q s^L ds\right)^{\frac{Q}{q}}.
	\end{equation}
	Now it follows from \eqref{2.5} and \eqref{2.6} that 
	\begin{align*}
	I &\geq |\wp| C^{\frac{Q}{q}}\left(\int_{0}^{1} |v(s)|^q s^L ds\right)^{\frac{Q}{q}} = |\wp|C^{\frac{Q}{q}} \left( \int_{0}^{R} \frac{|u(r)|^q}{r\left(\log \frac{Re}{r}\right)^{\alpha}}dr\right)^{\frac{Q}{q}}  \\
	& = |\wp|^{1-\frac{Q}{q}} C^{\frac{Q}{q}} \left( \int_{0}^{R} \frac{|u(x)|^q}{|x|^Q\left(\log \frac{Re}{|x|}\right)^{\alpha}}dx\right)^{\frac{Q}{q}}. 
	\end{align*}
	where $\alpha = \alpha(q,L)=\frac{Q-1}{Q}q+L+2$. The proof is complete. 
\end{proof}	

The method used in the previous section also allows one to obtain the following stability inequality for Rellich type inequalities:

\begin{proposition}\label{Rellichtype}
	Let $\mathbb{G}$ be a homogeneous group
	of homogeneous dimension $Q$. Let $|\cdot|$ be a homogeneous quasi-norm on $\mathbb{G}$ and $p\geq 1$. 
	Let $k \geq 2, k \in \mathbb{N}$ be such that $kp<Q$.
	Then for all real-valued radial functions $u\in C_0^{\infty}(\mathbb{G})$ we have 
	\begin{multline}
	\int_{\mathbb{G}} \frac{|\tilde{\mathcal{R}} u|^p}{|x|^{(k-2)p}} dx - K^p_{k,p} \int_{\mathbb{G}} \frac{|u|^p}{|x|^{kp}}dx \\ \geq C \sup_{R>0} 	\int_{\mathbb{G}}\frac{\left| |u(x)|^{\frac{p-2}{2}}u(x) -R^{\frac{Q-kp}{2}}|u(R)|^{\frac{p-2}{2}}u(R)|x|^{-\frac{Q-kp}{2}} \right|^2}{|x|^{kp}\left|\log \frac{R}{|x|}\right|^2} dx, 
	\end{multline}
	where  
	\begin{equation*}
	\tilde{\mathcal{R}}f = \mathcal{R}^{2}f + \frac{Q -1}{|x|} \mathcal{R}f
	\end{equation*}
	is the Rellich type operator on $\mathbb{G}$  and $K_{k,p} = \frac{(Q-kp)[(k-2)p+(p-1)Q]}{p^2}$.
\end{proposition}

\begin{proof}[Proof of Proposition \ref{Rellichtype}]
	For $k \geq 2, k \in \mathbb{N}$ and $kp<Q$ let us set
	\begin{equation}\label{4.2}
	v(r): = r^{\frac{Q-kp}{p}}u(r), \quad \text{where} \quad r \in [0, \infty). 
	\end{equation}
	Thus, $v(0)=0$ and $v(\infty) = 0$. 
	
	We have 
	\begin{align*}
	-\tilde{\mathcal{R}} u & =- \mathcal{R}^{2}\left(r^{\frac{kp-Q}{p}}v(r)\right) - \frac{Q-1}{r} \mathcal{R}\left(r^{\frac{kp-Q}{p}}v(r)\right)
	\\ &=- \mathcal{R}\left( \frac{kp-Q}{p}r^{\frac{kp-Q}{p}-1}v(r)
	+r^{\frac{kp-Q}{p}}\mathcal{R}v(r)\right)
	\\ &- \frac{Q -1}{r}\frac{kp-Q}{p}r^{\frac{kp-Q}{p}-1}v(r)
	-\frac{Q -1}{r}r^{\frac{kp-Q}{p}}\mathcal{R}v(r)
	\\ & = -\frac{kp-Q}{p}\left(\frac{kp-Q}{p}-1\right) r^{\frac{kp-Q}{p}-2}v(r)
	-\frac{kp-Q}{p}r^{\frac{kp-Q}{p}-1}\mathcal{R}v(r)
	\\ &- \frac{kp-Q}{p}r^{\frac{kp-Q}{p}-1}\mathcal{R}v(r)
	-r^{\frac{kp-Q}{p}}\mathcal{R}^{2}v(r)
	\\ &- \frac{Q -1}{r}\frac{kp-Q}{p}r^{\frac{kp-Q}{p}-1}v(r)
	-\frac{Q -1}{r}r^{\frac{kp-Q}{p}}\mathcal{R}v(r)
	\\ &=-r^{\frac{kp-Q}{p}-2}\left(\frac{(kp-Q)(kp-Q-p)}{p^{2}}+\frac{(Q-1)(kp-Q)}{p} \right)v(r)
	\\ &-r^{\frac{kp-Q}{p}-2}r^{2}\left(\mathcal{R}^{2}v(r)+\frac{1}{r}\left( \frac{2(kp-Q)}{p}+(Q-1)\right) \mathcal{R} v(r) \right)
	\\ & =  r^{k-2-\frac{Q}{p}}(K_{k,p}v(r)-r^2 \tilde{\mathcal{R}}_{k} v(r)),
	\end{align*}
	where  
	\begin{equation*}
	\tilde{\mathcal{R}}_{k}f = \mathcal{R}^{2}f + \frac{2k + \frac{Q(p-2)}{p} -1}{r} \mathcal{R}f
	\end{equation*}
	and $K_{k,p} = \frac{(Q-kp)[(k-2)p+(p-1)Q]}{p^2}$.
	By using the first inequality in Lemma \ref{ab_relation} with $a=K_{k,p}v(r)$ and $b = r^2 \tilde{\mathcal{R}}_{k} v(r)$, and the fact $\int_{0}^{\infty}|v|^{p-2}vv'dr = 0$ since $v(0)=0$ and $v(\infty) = 0$, we obtain
	\begin{align*}
	J &:= \int_{\mathbb{G}} \frac{|\tilde{\mathcal{R}} u|^p}{|x|^{(k-2)p}} dx - K^p_{k,p} \int_{\mathbb{G}} \frac{|u|^p}{|x|^{kp}}dx \\
	&= |\wp|\int_{0}^{\infty} |-\tilde{\mathcal{R}} u(r)|^p r^{Q-1-(k-2)p} dr - K^p_{k,p} |\wp|\int_{0}^{\infty} |u(r)|^p r^{Q-kp-1}dr \\
	&= |\wp|\int_{0}^{\infty} \left( |K_{k,p} v(r)-r^2\tilde{\mathcal{R}}_{k} v(r)|^p - (K_{k,p}v(r))^p \right)r^{-1}dr \\
	& \geq - p |\wp|K^{p-1}_{k,p} \int_{0}^{\infty} |v|^{p-2} v \tilde{\mathcal{R}}_{k} v r dr\\
	& = - p |\wp|K^{p-1}_{k,p} \int_{0}^{\infty} |v|^{p-2} v \left(v'' +\frac{2k + \frac{Q(p-2)}{p}-1}{r}v' \right)rdr \\
	& =- p |\wp|K^{p-1}_{k,p} \int_{0}^{\infty} |v|^{p-2} vv''rdr.
	\end{align*}
	On the other hand, we have
	\begin{align*}
	-\int_{0}^{\infty} |v|^{p-2} vv''rdr &= (p-1) \int_{0}^{\infty} |v|^{p-2} (v')^2 r dr + \int_{0}^{\infty} |v|^{p-2} vv'dr \\
	& = (p-1) \int_{0}^{\infty} |v|^{p-2} (v')^2 r dr
	\\ &= \frac{4(p-1)}{p^2} \int_{0}^{\infty} \left( 
	\frac{p-2}{2}\right)^{2} |v|^{p-2}(v')^2dr
	\\ &	+\frac{4(p-1)}{p^2} \int_{0}^{\infty}(p-2) |v|^{p-2}(v')^2+|v|^{p-2}(v')^2 r dr
	\\&= \frac{4(p-1)}{p^2} \int_{0}^{\infty} \left(\left(|v|^{\frac{p-2}{2}}\right)'v+|v|^{\frac{p-2}{2}}v'\right)^2 r dr \\
	&= \frac{4(p-1)}{p^2} \int_{0}^{\infty} |(|v|^{\frac{p-2}{2}}v)'|^2 r dr \\
	&= \frac{4(p-1)}{|\wp_{2}| p^2} \int_{\mathbb{G}_{2}} \left|  \mathcal{R}(|v|^{\frac{p-2}{2}}v) \right|^2 dx,
	\end{align*}
	where $\mathbb{G}_{2}$ is a homogeneous group of homogeneous degree $2$ and
	$|\wp_{2}|$ is the measure of the corresponding unit $2$-quasi-ball. 
	By using Lemma \ref{CKN} for $|v|^{\frac{p-2}{2}}v \in C^{\infty}_{0} (\mathbb{G}_{2} \backslash \{0\})$ in $p=Q=2$ case, and combining above equalities, we obtain 
	\begin{align*}
	J &\geq C_{1} \int_{\mathbb{G}_{2}} \frac{\left| |v(x)|^{\frac{p-2}{2}} v(x) - |v(R\frac{x}{|x|})|^{\frac{p-2}{2}} v(R\frac{x}{|x|} )\right|^2}{|x|^2 \left|\log \frac{R}{|x|}\right|^2} dx  \\
	& = C_{1} \int_{0}^{\infty} \frac{\left| |v(r)|^{\frac{p-2}{2}} v(r) - |v(R)|^{\frac{p-2}{2}} v(R) \right|^2}{r \left|\log \frac{R}{r}\right|^2} dr \\
	& = C_{1} \int_{0}^{\infty} \frac{\left| |u(r)|^{\frac{p-2}{2}} u(r) - R^{\frac{Q-kp}{2}}|u(R)|^{\frac{p-2}{2}} u(R) r^{-\frac{Q-kp}{2}} \right|^2}{r^{1-Q+kp} \left|\log \frac{R}{r}\right|^2} dr
	\end{align*} 
	for any $R>0$. That is, 
	\begin{align*}
	J \geq C \sup_{R>0} \int_{\mathbb{G}} \frac{\left| |u(x)|^{\frac{p-2}{2}}u(x) -R^{\frac{Q-kp}{2}}|u(R)|^{\frac{p-2}{2}}u(R)|x|^{-\frac{Q-kp}{2}} \right|^2}{|x|^{kp}\left|\log \frac{R}{|x|}\right|^2} dx. 
	\end{align*}
	The proof is complete.
\end{proof}

{\bf Acknowledgments.} The authors were supported in parts by the EPSRC
grant EP/K039407/1 and by the Leverhulme Grant RPG-2014-02,
as well as by the MESRK grant 5127/GF4. No new data was collected or
generated during the course of research.

\bibliographystyle{amsplain}

\end{document}